\documentclass[10pt,reqno]{amsart}
\usepackage{amsmath}
\usepackage{amssymb}
\usepackage{amsthm}

\newlength{\defbaselineskip}
\newcommand{\setlinespacing}[1]%
           {\setlength{\baselineskip}{#1 \defbaselineskip}}

\numberwithin{equation}{section}

\newtheorem{thm}{Theorem}[section]

\newtheorem{lem}[thm]{Lemma}
\newtheorem{prop}[thm]{Proposition}

\theoremstyle{definition}

\theoremstyle{remark}
\newtheorem{rem}[thm]{Remark}
\numberwithin{equation}{section}

\begin{document}

\title[Eigenvalue bounds]
{A note on eigenvalue bounds for Schr\"odinger operators}

\author{Yoonjung Lee and Ihyeok Seo}

\thanks{I. Seo was supported by the NRF grant funded by the Korea government(MSIP) (No. 2017R1C1B5017496).}


\address{Department of Mathematics, Sungkyunkwan University, Suwon 16419, Republic of Korea}
\email{yjglee@skku.edu}

\address{Department of Mathematics, Sungkyunkwan University, Suwon 16419, Republic of Korea}
\email{ihseo@skku.edu}

\begin{abstract}
We obtain a new bound on the location of eigenvalues for a non-self-adjoint Schr\"{o}dinger operator
with complex-valued potentials by obtaining a weighted $L^2$ estimate for the resolvent of the Laplacian.
\end{abstract}

\maketitle

\section{Introduction}
In this paper we are concerned with bounds on the location of eigenvalues of a non-self-adjoint Schr\"{o}dinger operator $-\Delta +V(x)$
in $\mathbb{R}^d$
with complex-valued potentials $V$ in terms of the size of them.

When dimension $d=1$ it turns out that every eigenvalue $\lambda \in \mathbb{C}\setminus[0,\infty)$ satisfies
\begin{equation}\label{dim1}
|\lambda|^{1/2} \leq \frac12\int_{\mathbb{R}} |V(x)|dx
\end{equation}	
which shows that $\lambda$ lies in a disk whose radius is controlled by the $L^1$-norm of $V$ (see \cite{AAD,DN}).
There have been also attempts \cite{FLLS,LS,S} to obtain this type of bounds in higher dimensions but away from the positive half-axis $[0,\infty)$.
In {\cite{LS}}, the following natural extension of \eqref{dim1},
\begin{equation}\label{conj}
|\lambda|^{\gamma} \leq C_{d, \gamma} \int_{\mathbb{R}^d} |V(x)|^{\gamma + d/2} dx,
\end{equation}	
was left unsolved for $d\geq2$ and $0<\gamma\leq d/2$.
In fact, if $V$ is real-valued, it follows from Sobolev inequalities that \eqref{conj} holds
for $\gamma\geq1/2$ if $d=1$ and $\gamma>0$ if $d\geq2$ (see \cite{K,LT}).
However, the problem of obtaining \eqref{conj} becomes much more difficult if $V$ is allowed to be complex-valued.

It is a remarkable observation of Frank \cite{F} that the following resolvent estimate\footnote{In \cite{KRS} this estimate is only proved for $d\geq3$, but the same argument works for $d=2$ as well.} due to Kenig-Ruiz-Sogge \cite{KRS}
can be used to obtain the bound \eqref{conj}
based on the Birman-Schwinger principle:
\begin{equation}\label{urs}
\|(-\Delta -z)^{-1}f\|_{L^{p'}} \leq C |z|^{-\frac{d}{2}+\frac{d}{p}-1} \|f\|_{L^p},
\end{equation}
where $2d/(d+2) < p \leq 2(d+1)/(d+3)$ and $d\geq2$.
Indeed, the range of $p$ in \eqref{urs} implies \eqref{conj} whenever $0<\gamma\leq 1/2$.
It should be noted here that \eqref{urs} cannot hold for $2(d+1)/(d+3)<p<2d/(d+1)$
corresponding to $1/2<\gamma<d/2$ in \eqref{conj}.
This means that \eqref{conj} cannot be obtained for $\gamma>1/2$ using $L^p-L^{p'}$ resolvent estimates.
Recently, a different replacement of the bound \eqref{conj} for the remaining case $\gamma>1/2$ was given by Frank \cite{F2} again:
$$\delta(\lambda)^{\gamma-1/2}|\lambda|^{1/2}\leq C_{\gamma,d}\int_{\mathbb{R}^d}|V|^{\gamma+d/2}dx$$
which is weaker than \eqref{conj} since $\delta(\lambda):=\textrm{dist} (\lambda,[0,\infty))\leq|\lambda|$.
See also \cite{FS} for another replacement.

By the way, the class $L^p$ is too small to contain the inverse square potential
$V(x)=1/|x|^2$ which has attracted considerable interest
from mathematical physics. This is because the Schr¡§odinger operator $-\Delta+ 1/|x|^2$ is
physically related to the Hamiltonian of a spin-zero quantum particle in a Coulomb
field (\cite{C}).
In this regard, we shall consider from now on a wider class of potentials where one can consider stronger singularities
of the type $1/|x|^2$.

For this we first need to introduce the Kerman-Saywer class $\mathcal{KS}_{\alpha}$ which is defined for $0<\alpha<d$\, if
$$\|V\|_{\mathcal{KS}_{\alpha}} := \sup_Q \left( \int_Q |V(x)| dx \right)^{-1} \int_Q \int_Q \frac{|V(x) V(y)|}{|x-y|^{d-\alpha}}  dxdy < \infty,$$
where the sup is taken over all dyadic cubes $Q$ in $\mathbb{R}^d$.
This class particularly when $\alpha=2$ is closely related to the global Kato and Rollnik classes which are fundamental in spectral and scattering theory and satisfy
$$\|V\|_{\mathcal{K}} := \sup_{x \in \mathbb{R}^n} \int_{\mathbb{R}^n} \frac{|V(y)|}{|x-y|^{n-2}} dy < \infty$$
and
$$||V||_{\mathcal{R}} := \int_{\mathbb{R}^3} \int_{\mathbb{R}^3} \frac{|V(x) V(y)|}{|x-y|^2} dxdy < \infty,$$
respectively.
Indeed, $\mathcal{K}\subset\mathcal{KS}_2$ and $\mathcal{R}\subset\mathcal{KS}_2$.
Also, $\mathcal{KS}_{\alpha}$ is wider than the Morrey-Campatano class $\mathcal{L}^{\alpha,p}$
which is defined for $\alpha>0$ and $1\leq p\leq d/\alpha$ by
$$\|f\|_{\mathcal{L}^{\alpha,p}} := \sup_{x, r} r^{\alpha} \left( r^{-d} \int |f(x)|^p dx\right)^{\frac{1}{p}} < \infty.$$
In general,
$1/|x|^\alpha\subset L^{d/\alpha,\infty} \subset \mathcal{L}_{\alpha, p} \subset \mathcal{KS}_{\alpha}$ for $1<p<d/\alpha$
(see \cite{BBRV}, Subsection 2.2).

Our result below gives bounds like \eqref{conj} for the eigenvalues of the Schr\"odinger operator
in terms of the size $\|V^\beta\|^{1/\beta}_{\mathcal{KS}_{\alpha}}$ with $\beta=(2\alpha-d+1)/2$.

\begin{thm}\label{thm}
Let	$d\geq2$. If $d-1\leq\alpha<d$ for $d\geq3$ and if $3/2\leq\alpha<2$ for $d=2$,
then any eigenvalue $\lambda\in\mathbb{C}\setminus[0,\infty)$
of the Schr\"odinger operator $-\Delta+V(x)$ satisfies
\begin{equation}\label{bound}
	 |\lambda|^{\frac{\alpha-d+1}{2\alpha-d+1}}\leq C
\|V^{\beta}\|_{\mathcal{KS}_{\alpha}} ^{1/\beta}
	\end{equation}
where $\beta=(2\alpha-d+1)/2$.	
\end{thm}

\begin{rem}\label{rem}
Particularly when $\beta=1$, $\alpha=2$ if $d=3$ and $\alpha=3/2$ if $d=2$.
In this case, \eqref{bound} becomes
	\begin{equation*}
	|\lambda|^{1/4} \leq C\|V\|_{\mathcal{KS}_{3/2}}
	 \end{equation*}
when $d=2$, and
\begin{equation} \label{2dim}
	1\leq C\|V\|_{\mathcal{KS}_{3/2}}
	 \end{equation}
when $d=3$.
In particular, the bound \eqref{2dim} has the following consequence:
if $\|V\|_{\mathcal{KS}_{3}}$ is sufficiently small, then the Schr\"odinger operator has no eigenvalue in $\mathbb{C}\setminus[0,\infty)$.
\end{rem}

\begin{rem}\label{rem2}
Our theorem considers singularities of the type $a/|x|^2$ with
a small $a>0$ (this smallness condition is a natural restriction - see \cite{RS}, p. 172).
Indeed, when $2\beta=\alpha$, the condition $\beta=(2\alpha-d+1)/2$ becomes $\alpha=d-1$,
which implies
  \begin{equation*}
	 1\leq Ca\||x|^{-\alpha}\|_{\mathcal{KS}_\alpha} ^{1/\beta}\leq Ca
	\end{equation*}
when applying \eqref{bound} with $d\geq3$ to $a/|x|^2$.
This means that the Shr\"odinger operator $-\Delta+a/|x|^2$ with a small $a>0$ has no eigenvalue in $\mathbb{C}\setminus[0,\infty)$.
\end{rem}

\begin{rem}
It should be noted that the above theorem improves the following previous result due to Frank \cite{F} that
\begin{equation}\label{bound22}
	 |\lambda|^{\frac{2\gamma}{2\gamma+d}}\leq C
\|V\|_{\mathcal{L}^{2d/(2\gamma+d),p}}
\end{equation}
for $0<\gamma<1/2$ and $\frac{(d-1)(2\gamma+d)}{2(d-2\gamma)}<p\leq\gamma+\frac d2$ (see Theorem 3 there).
Indeed, if $V\in\mathcal{L}^{\delta,p}$, then $V^\beta\in\mathcal{L}^{\delta\beta,p/\beta}\subset\mathcal{KS}_{\delta\beta}$
when $p>\beta$, and
$$\|V^\beta\|^{1/\beta}_{\mathcal{KS}_{\delta\beta}}\leq\|V^\beta\|^{1/\beta}_{\mathcal{L}^{\delta\beta,p/\beta}}
=\|V\|_{\mathcal{L}^{\delta,p}}.$$
Using this fact and setting $\alpha=\delta\beta$ with $\delta=2d/(2\gamma+d)$ and $\beta=(2\alpha-d+1)/2$,
one can easily check that \eqref{bound} implies \eqref{bound22}, as follows:
\begin{equation*}
	 |\lambda|^{\frac{2\gamma}{2\gamma+d}}= |\lambda|^{\frac{\alpha-d+1}{2\alpha-d+1}}\leq C
\|V^{\beta}\|_{\mathcal{KS}_{\alpha}} ^{1/\beta}\leq C
\|V\|_{\mathcal{L}^{2d/(2\gamma+d),p}}
	\end{equation*}
under the same conditions on $\gamma$ and $p$ as in \eqref{bound22}.
\end{rem}

The rest of this paper is organized as follows:
In Section \ref{sec2} we prove Theorem \ref{thm} based on the Birman-Schwinger principle.
In order to make use of this principle, we obtain weighted $L^2$ resolvent estimates with weights in the class $\mathcal{KS}_\alpha$.
The key point here is that the integral kernel of $(-\Delta-z)^{-\zeta}$ can be controlled by
that of the fractional integral operator which will be shown in the final section, Section \ref{sec3}.
This observation leads us to apply the known weighted estimates for this operator to obtain the desired resolvent estimates.

\

\noindent\textbf{Acknowledgments.}
We are grateful to David Krej\v{c}i\v{r}i\'{\i}k for bringing our attention to the related papers~\cite{FKV,FKV2}.

\section{Proof of Theorem \ref{thm}}\label{sec2}

Following \cite{F}, we shall use the following Birman-Schwinger principle:
if $\lambda\in\mathbb{C}\setminus[0,\infty)$ is an eigenvalue of the Schr\"odinger operator
$-\Delta+V$, then $-1$ is an eigenvalue of the Birman-Schwinger operator
$$A:=V^{1/2}(-\Delta-\lambda)^{-1}|V|^{1/2},$$
where $V^{1/2}=\frac{V}{|V|}|V|^{1/2}$.

This principle implies that the norm of this operator in $L^2$ is at least 1.
Indeed,
$$\|A\|=\sup_{\|\psi\|_2\leq1}\frac{\|A\psi\|_2}{\|\psi\|_2}\geq\sup_{\|\psi_\lambda\|_2\leq1}\frac{\|\lambda\psi_\lambda\|_2}{\|\psi_\lambda\|_2}
=\sup_{\|\psi_\lambda\|_2\leq1}|\lambda|\geq|-1|=1,$$
where $\psi_\lambda$ is an eigenfunction corresponding to an eigenvalue $\lambda$.
Hence the proof of \eqref{bound} is now reduced to showing that
\begin{equation*}
\|V^{1/2}(-\Delta-\lambda)^{-1}|V|^{1/2}\|\leq  C|\lambda|^{-\frac{\alpha-d+1}{2\alpha-d+1}}
\||V|^{\beta}\|_{\mathcal{KS}_{\alpha}} ^{\frac{1}{\beta}}
\end{equation*}
which follows immediately from the following resolvent estimate:

\begin{prop}\label{LT}
Let	$d\geq2$. If $d-1\leq\alpha<d$ for $d\geq3$ and if $3/2\leq\alpha<2$ for $d=2$, then
\begin{equation}\label{lt_ks}
	\||V|^{\frac{1}{2}}(-\Delta-\lambda)^{-1}|V|^{\frac{1}{2}} f\|_{L^2} \leq C|\lambda|^{-\frac{\alpha-d+1}{2\alpha-d+1}}
\||V|^{\beta}\|_{\mathcal{KS}_{\alpha}} ^{\frac{1}{\beta}} \|f\|_{L^2}
	\end{equation}
where $\beta=(2\alpha-d+1)/2$.	
\end{prop}

\begin{proof}[Proof of Proposition \ref{LT}]
We only consider the case where $d\geq3$ because the case $d=2$ would follow obviously from the same argument.

We will use Stein's complex interpolation ({\it cf. \cite{SW}}).
Following \cite{KRS,CS},
we first consider the analytic family of operators,
$${T}_{\zeta} = |V|^{\frac{\zeta}{2}}(-\Delta-\lambda)^{-\zeta}|V|^{\frac{\zeta}{2}},$$
with $0\leq \text{Re}\,\zeta\leq(d+1)/2$.
Then we shall obtain \eqref{lt_ks} with the assumption that $|\lambda|=1$, as follows:
\begin{equation}\label{rce}
\||V|^{\frac{1}{2}}[(-\Delta-\lambda)^{-1}|V|^{\frac{1}{2}}f]\|_{L^2} \leq C\||V|^{\beta}\|_{\mathcal{KS}_\alpha} ^{\frac{1}{\beta}}\|f\|_{L^2}
\end{equation}
under the same conditions on $\alpha$ and $\beta$ as in \eqref{lt_ks}.
Once we obtain \eqref{rce}, the desired estimate \eqref{lt_ks} follows from the scaling.
Indeed, note first that
$$R(\lambda)f(x):=(-\Delta-\lambda)^{-1}f(x)=|\lambda|^{-1}(-\Delta-\lambda/|\lambda|)^{-1}[f(|\lambda|^{-1/2}\cdot)](|\lambda|^{1/2}x).$$
So, if we have \eqref{rce} with $|\lambda|=1$, then
\begin{align*}
\|R(\lambda)f\|_{L^2(|V|)}^2&=|\lambda|^{-2}|\lambda|^{-d/2}\|R(\lambda/|\lambda|)[f(|\lambda|^{-1/2}\cdot)]\|_{L^2(|V(|\lambda|^{-1/2}\cdot)|)}^2\\
&\leq C|\lambda|^{-2}|\lambda|^{-d/2}\||V(|\lambda|^{-1/2}\cdot)|^\beta\|_{\mathcal{KS}_\alpha} ^{\frac{2}{\beta}}\|f(|\lambda|^{-1/2}\cdot)\|_{L^2(|V(|\lambda|^{-1/2}\cdot)|^{-1})}^2\\
&\leq C|\lambda|^{-2}\||V(|\lambda|^{-1/2}\cdot)|^\beta\|_{\mathcal{KS}_\alpha} ^{\frac{2}{\beta}}\|f\|_{L^2(|V|^{-1})}^2.
\end{align*}
On the other hand,
\begin{align*}
\||V_\lambda |^{\beta}\|_{\mathcal{KS}_{\alpha}}
&=\sup_{Q} \biggl(\int_{Q} | V_\lambda |^{\beta} dx\biggl)^{-1}
\int_{Q}\int_{Q} \frac{|V_\lambda|^{\beta} (x)|V_\lambda|^{\beta} (y) }{ |x-y|^{d-\alpha}} dx dy \\
&=  \sup_{Q} \biggl( \int_{Q} |V|^{\beta}(x) |\lambda|^{\frac{d}{2}} dx \biggl)^{-1} \int_{Q} \int_{Q} \frac{|V|^{\beta}(x)|V|^{\beta}(y)|}{|\lambda|^{(d-\alpha)/2}|x-y|^{d-\alpha}} |\lambda|^{d} dx dy  \\
&=|\lambda|^{\frac{\alpha}{2}}|| |V|^{\beta} ||_{\mathcal{KS}_\alpha}
\end{align*}
where $V_\lambda(\cdot)=V(|\lambda|^{-1/2}\cdot)$.
Consequently, we get
$$\|R(\lambda)f\|_{L^2(|V|)}
\leq C|\lambda|^{\frac{\alpha}{2\beta}-1}\||V|^\beta\|_{\mathcal{KS}_\alpha} ^{\frac{1}{\beta}}\|f\|_{L^2(|V|^{-1})}$$
which is equivalent to \eqref{lt_ks}.

Now we show \eqref{rce}. By Stein's interpolation we only need to show the following two estimates
\begin{equation}\label{ce1}
\|{T}_{\zeta} f\|_{L^2}\leq\, C e^{c|\text{Im}\,\zeta|}\|f\|_{L^2} \quad \text{for}\quad \text{Re}\,\zeta = 0
\end{equation}
and
\begin{equation}\label{ce2}
\|{T}_{\zeta} f\|_{L^2}\leq\, C e^{c|\text{Im}\,\zeta|}\||V|^{\text{Re}\,\zeta}\|_{\mathcal{KS}_\alpha}\|f\|_{L^2}
\end{equation}
for $\frac{d-1}{2}\leq\text{Re}\,\zeta<\frac{d+1}{2}$ and $\alpha= \frac{d-1}{2}+\text{Re}\,\zeta$.
Indeed, by interpolating these estimates one can see
$$\|{T}_{1} f\|_{L^2}\leq\, C \||V|^{\text{Re}\,\zeta}\|_{\mathcal{KS}_\alpha}^{1/\text{Re}\,\zeta}\|f\|_{L^2}$$
for $\frac{d-1}{2}\leq\text{Re}\,\zeta<\frac{d+1}{2}$ and $\alpha= \frac{d-1}{2}+\text{Re}\,\zeta$.
By setting $\beta=\text{Re}\,\zeta$, it is now easy to see that this is equivalent to the estimate \eqref{rce} as desired.

Finally we show \eqref{ce1} and \eqref{ce2}.
The first estimate follows easily from Plancherel's theorem.
In fact, since $|V|^{\zeta/2}=1$ for $\text{Re}\,\zeta=0$, we get
\begin{align*}
\|T_\zeta f\|_{L^2}&=\bigg\|\frac1{(|\xi|^2-\lambda)^\zeta}\widehat{f}(\xi)\bigg\|_{L^2}\\
&\leq\sup_{\xi\in\mathbb{R}^n}\bigg|\frac1{(|\xi|^2-\lambda)^\zeta}\bigg|\|f\|_{L^2}\\
&\leq\sup_{\xi\in\mathbb{R}^n}e^{\text{Im}\,\zeta\arg(|\xi|^2-\lambda)}\|f\|_{L^2}\\
&\leq e^{\pi|\text{Im}\,\zeta|}\|f\|_{L^2}
\end{align*}
by Plancherel's theorem.
It remains to show \eqref{ce2}.
The key point here is that the integral kernel $K_{\zeta}$ of $(-\Delta-\lambda)^{-\zeta}$ with $|\lambda|=1$ can be controlled
by that of the fractional integral operator $I_\alpha$
which is defined for $0<\alpha<d$ by
$$I_\alpha f(x)=\int_{\mathbb{R}^d}\frac{f(y)}{|x-y|^{d-\alpha}}dy,$$
as follows:
\begin{equation}\label{ks}
|K_\zeta(x)| \leq C\ e^{c|\text{Im}\,\zeta|} |x|^{(\frac{d-1}{2} +\text{Re}\,\zeta)-d}
\end{equation}
for $\frac{d-1}{2}\leq\text{Re}\,\zeta\leq\frac{d+1}{2}$,
which will be shown later in the next section.

Now we use the following lemma,
which gives the characterization of the weighted $L^2$ estimates for the fractional integrals,
due to Kerman and Sawyer \cite{KS} (see Theorem 2.3 there and also Lemma 2.1 in \cite{BBRV}):

\begin{lem}
Let $0<\alpha<d$. Suppose that $w$ is a nonnegative measurable function on $\mathbb{R}^d$.
Then there exists a constant $C_w$ depending on $w$ such that $w\in\mathcal{KS}_\alpha$ if and only if the following two equivalent estimates
\begin{equation*}
\|I_{\alpha/2}f\|_{L^2(w)}\leq C_w\|f\|_{L^2}
\end{equation*}
and
$$\|I_{\alpha/2}f\|_{L^2}\leq C_w\|f\|_{L^2(w^{-1})}$$
are valid for all measurable functions $f$ on $\mathbb{R}^d$.
Furthermore, the constant $C_w$ may be taken to be a constant multiple of $\|w\|_{\mathcal{KS}_\alpha}^{1/2}$.
\end{lem}

First, note that the two estimates in the lemma directly implies
\begin{equation*}
\big\|w^{1/2}I_\alpha(w^{1/2}f)\big\|_{L^2}\leq C\|w\|_{\mathcal{KS}_\alpha}\|f\|_{L^2}.
\end{equation*}
Using \eqref{ks} and this estimate with $w=|V|^{\text{Re}\,\zeta}$ and $\alpha=(d-1)/2+\text{Re}\,\zeta$,
we get the desired estimate \eqref{ce2}.

\end{proof}

\section{Appendix}\label{sec3}
In this final section we provide a proof of the estimate \eqref{ks} for the integral kernel $K_\zeta$ of $(-\Delta-\lambda)^{-\zeta}$
with $|\lambda|=1$.
Its detailed proof for $\text{Re}\,\zeta=(d-1)/2$ can be also found in \cite{Se}.

To show \eqref{ks}, we first recall the following formula for $K_{\zeta}$ ({\it cf. \cite{GS}} pp.\,288-289, \cite{KRS}):
\begin{equation*}
K_{\zeta}(x)= \frac{e^{\zeta^2}2^{-\zeta+1}}{(2\pi)^{d/2}\Gamma(\zeta)\Gamma(\frac{d}{2}-\zeta)}
\bigg( \frac{\lambda}{|x|^2}\bigg)^{\frac12(\frac{d}2-\zeta)}B_{d/2-\zeta}\ (\sqrt{\lambda|x|^2})
\end{equation*}
where $B_{\nu}(w)$ is the Bessel kernel of the third kind which satisfies for $\text{Re}\,w>0$
\begin{equation}\label{b1}
|e^{\nu^2}\nu B_{\nu}(w)|\leq C|w|^{-|\text{Re}\,\nu|},\quad|w|\leq1,
\end{equation}
and
\begin{equation}\label{b2}
|B_{\nu}(w)|\leq C_{\text{Re}\,\nu}e^{-\text{Re}\,w}|w|^{-1/2},\quad|w|\geq1.
\end{equation}
See \cite{KRS}, p. 339 for details.
To estimate $|K_\zeta|$ using these estimates with $\nu=d/2 -\zeta$ and $w=\sqrt{\lambda|x|^2}$, we then note that $\text{Re}(\sqrt{\lambda|x|^2})=|x|\cos(\frac12\arg\lambda)>0$ for $x\neq0$,
since $-\pi<\arg\lambda\leq\pi$ by the principal branch, and $\lambda\not\in\mathbb{R}$.
Also, the inequalities $|\sqrt{\lambda|x|^2}| \ge1$ and $|\sqrt{\lambda|x|^2}|\leq 1$ correspond to $|x|\ge1$ and $|x|\leq 1$, respectively,
since we are assuming $|\lambda|=1$.

It follows now from $\eqref{b2}$ that for $|x|\geq1$
\begin{align*}
|K_\zeta(x)|
& \leq \bigg|\frac{e^{\zeta^2} 2^{-\zeta+1}\ }  {(2\pi)^{d/2}\ \Gamma(\zeta) \Gamma(\frac{d}{2}-\zeta) }
\bigg( \frac{\lambda}{|x|^2}\bigg)^{\frac12(d/2-\zeta)}\bigg|e^{ -\text{Re}\sqrt{\lambda|x|^2}}\big|\sqrt{\lambda|x|^2}\big|^{-\frac{1}{2}}\\
&\leq
Ce^{ -\text{Re}\sqrt{\lambda|x|^2}}\bigg|\bigg( \frac{\lambda}{|x|^2}\bigg)^{\frac12(d/2-\zeta)}\bigg|\big|\sqrt{\lambda|x|^2}\big|^{-\frac{1}{2}}\\
&\leq
Ce^{-|x|\cos(\frac12\arg\lambda)}\bigg|\bigg( \frac{\lambda}{|x|^2}\bigg)^{\frac12(d/2-\zeta)}\bigg|\big|\sqrt{\lambda|x|^2}\big|^{-\frac{1}{2}}\\
&\leq Ce^{-|x|\cos(\frac12\arg\lambda)}e^{\frac{1}{2}\text{Im}\,\zeta \arg \lambda} |x|^{\text{Re}\,\zeta-\frac{d+1}{2}} \\
&\leq C e^{\frac{\pi}2 |\text{Im}\,\zeta|}|x|^{\text{Re}\,\zeta-\frac{d+1}{2}}
\end{align*}
as desired.
(Here we used the fact that $|\lambda|=1$ and $-\pi<\arg\lambda\leq\pi$.)
On the other hand, when $|x|\leq1$, it follows from $\eqref{b1}$ that
\begin{align}\label{inq}
\nonumber|K_\zeta(x)|& \leq \bigg| \frac{e^{\zeta^2} 2^{-\zeta+1} e^{-(d/2-\zeta)^2} } {(2\pi)^{d/2} \Gamma(\zeta) \Gamma(\frac{d}{2}-\zeta)(d/2-\zeta)}
\bigg( \frac{\lambda}{|x|^2}\bigg)^{\frac{1}{2}(\frac{d}{2}-\zeta)} \bigg|\big| \sqrt{\lambda|x|^2}\big|^{-| \text{Re}( \frac{d}{2}-\zeta ) |} \\
\nonumber&\leq
C\bigg|  \bigg( \frac{\lambda}{|x|^2}\bigg)^{\frac{1}{2}(\frac{d}{2}-\zeta)} \bigg|\big| \sqrt{\lambda|x|^2}\big|^{-| \text{Re}( \frac{d}{2}-\zeta ) |} \\
\nonumber&\leq C e^{\frac{1}{2}\text{Im}\,\zeta \arg \lambda}     |x|^{-\frac{d}{2}+\text{Re}\,\zeta -|\text{Re}(\frac{d}{2}-\zeta)|} \\
&\leq C e^{\frac{\pi}2 |\text{Im}\,\zeta|} |x|^{-\frac{d}{2}+Re\zeta -| \frac{d}{2}-\text{Re}\,\zeta|}.
\end{align}
When $\text{Re}\,\zeta\geq d/2$, the estimate \eqref{inq} implies
\begin{align*}
|K_\zeta(x)|
&\leq C e^{\frac{\pi}2 |\text{Im}\,\zeta|}\\
&\leq C e^{\frac{\pi}2 |\text{Im}\,\zeta|} |x|^{\text{Re}\,\zeta-\frac{d+1}{2}}
\end{align*}
since $|x|^{\text{Re}\,\zeta -\frac{d+1}{2} } \ge 1$ if $ \text{Re}\, \zeta\leq(d+1)/2 $ and $|x| \leq 1$.
When $\text{Re}\,\zeta \leq d/2$, from \eqref{inq} we have
\begin{align*}
|K_\zeta(x)|
&\leq  Ce^{\frac{\pi}2 |\text{Im}\,\zeta|}|x|^{-2(d/2-\text{Re}\,\zeta)} \\
&\leq  Ce^{\frac{\pi}2 |\text{Im}\,\zeta|} |x|^{\text{Re}\,\zeta-\frac{d+1}{2}}
\end{align*}
if $\text{Re}\,\zeta\geq(d-1)/2$ and $|x| \leq 1$.
Consequently, we get for $|x|\leq 1$
$$|K_\zeta(x)|\leq  Ce^{\frac{\pi}2 |\text{Im}\,\zeta|} |x|^{\text{Re}\,\zeta-\frac{d+1}{2}}$$
if $(d-1)/2 \leq \text{Re}\, \zeta \leq (d+1)/2$, as desired.

\end{document}